\documentclass[12pt]{article}
\usepackage{mathrsfs}
\usepackage{amsfonts,amsthm, amssymb}
\usepackage[tbtags]{amsmath}
\usepackage{latexsym, euscript, epic, eepic}
\usepackage{lineno}
\usepackage{cite}
\usepackage[citecolor=blue]{hyperref}
\newtheorem{cro}{Corollary}[section]
\newtheorem{defn}{Definition}[section]
\newtheorem{prop}{Proposition}[section]
\newtheorem{thm}{Theorem}[section]
\newtheorem{lem}{Lemma}[section]
\newtheorem{Example}{Example}[section]
\newtheorem{rem}{\bf Remark}[section]
\numberwithin{equation}{section}

\begin{document}
\title{Non-dense orbits on topological dynamical systems
 \footnotetext {* Corresponding author}
  }
\author{Cao Zhao$^{1}$, Jiao Yang$^{2}$ and Xiaoyao Zhou$^{*2}$\\
	\small 1   School of Mathematical Sciences, Suzhou University of Science and Technology\\
	 \small   Nanjing 210059, Jiangsu, P.R.China\\
  \small   2 School of Mathematical Sciences and Institute of Mathematics, Nanjing Normal University,\\
   \small   Nanjing 210046, Jiangsu, P.R.China\\
 \small    e-mail: zhaocao@usts.edu.cn, jiaoyang6667@126.com, zhouxiaoyaodeyouxian@126.com
}
\date{}
\maketitle

\begin{center}
 \begin{minipage}{120mm}
{\small {\bf Abstract.}
Let  $(X,d,T )$ be  a topological dynamical system with the specification property. We consider the non-dense orbit set $E(z_0)$ and show that for any non-transitive point $z_0\in X$, this set $E(z_0)$ is empty  or carries full     topological pressure.  }
\end{minipage}
 \end{center}

\vskip0.5cm
\section{Introduction}
A number $\theta $ is called badly approximable if $ |\theta-p/q |> c/q^2$ for some $c>0$ and all rational numbers $p/q$. It is well known that the set of badly approximable numbers has Lebesgue measure zero. In 1930s, Jarn\'{i}k \cite{Jar} showed that the set of badly approximable numbers is of full Hausdorff dimension. In 1960s, Schmidt introduced a game, which is known as Schmidt's game, and proved that the set of badly approximable numbers is 1/2-winning in the sense of the game \cite{Sch}.
The winning property implies density and full Hausdorff dimension.
 In \cite{Abe}, Abercrombie and Nair showed the non-dense set for  the expanding rational map of the Riemann sphere acting on
its Julia set $\mathbf J$ has full Hausdorff dimension. And this result can be seen as an opposite result of Hill and Velani (\cite{Hil}, \cite{Hil1}, \cite{Hil2}).
In fact, \cite{Kris} and \cite{Sara} generalized the result of \cite{Abe} to some more general systems.
Similarly, this result holds for many systems,
such as toral endomorphisms (\cite{Brod, Dani}), transitive $C^2$-Anosov diffeomorphisms (\cite{Dol, Urb}) and partially hyperbolic diffeomorphisms \cite{Wu, Wu1 }.
Recently,  Tseng proved that for a $C^2$-expanding endomorphism on the circle the non-dense set is a winning set for Schmidt's game and asked a question whether there are non-algebraic dynamical systems with winning non-dense set in dimensions greater than one (\cite{Tseng}). Later, he answered this question and proved that a certain Anosov diffeomorphism on the $2$-torus has a winning non-dense set, by using the $C^1$ conjugacy between such system and a linear hyperbolic automorphism on the $2$-torus (\cite{Tseng1}). In \cite{Wu2},  Wu studied the non-algebraic dynamical systems in dimensions greater than one have non-dense sets winning for Schmidt's game.

In view of our main results, it is worth mentioning the following point of view of Jarn\'{i}k's theorem and its generalizations. Let $(X, d, T)$ be a topological dynamical system, where $(X,d)$ is a compact metric space and $T:X\to X$ is a continuous map.
For any $x\in X$, let  $O_T(x)$ denote the orbit of $x$, i.e.,
$O_T(x):=\{x, Tx ,\cdots, T^nx\cdots \}$.
In fact, for any  $z_0 \in X$, we define the non-dense orbit set as follows
$$E(z_0 )= \{x\in X:~ z_0\notin \overline{\{T^n (x):~n\geq 0\}}\},$$
where  $x\in E(z_0)$ means $z_0$ is badly approximated by the orbit of $x$. When $T$ is Gauss map, $E(0)$ is the set of badly approximable numbers.
From the definition, any point in $E(z_0)$ has a non-dense forward orbit in $X$.

While the non-dense orbit set has been extensively studied for so many special systems, the general topological dynamical systems is rarely considered. Topological entropy, as a Hausdorff dimension,  plays an important role in topological dynamical systems. Topological entropy is an invariant of a dynamical system which describes the complexity of the orbits of the systems. The definition of topological entropy for compact invariant set was first given by Adler, Konheim and McAndrew \cite{Mca}. Bowen \cite{Bow} extend it to non-compact invariant sets and noticed the dimensional nature of this notion. Pesin and Pitskel \cite{PesPit} generalized the Bowen topological entropy on non-compact subsets to topological pressure.
In this paper, we try to study the topological pressure of non-dense orbits on the general topological dynamical systems.

 In this paper, we will study the subset $E(z_0  )\subset X$ in topological dynamical systems with specifications property (See Section \ref{sep} for precise definitions.)

Now we state our main results as follows.
\begin{thm}\label{main}
 Suppose $(X, d, T)$ be a dynamical system with the specification property. For any non-transitive point $z_0\in X$, i.e. $\overline{O_T(z_0)}\neq X$ and $\varphi \in C(X)$, then either $E(z_0)=\emptyset$ or 
$$P(E(z_0  ),\varphi)=P_{top}(\varphi),$$
where $P(Z,\varphi)$ denotes topological pressure of non-compact subset $Z$ and $P_{top}(\varphi)$ denotes the classical topological pressure of $X$.
\end{thm}

\section{Preliminaries}\label{sep}
Let $(X, d)$ be a compact metric space and $T: X\to X$ be a continuous map. Let $C(X)$ denote the Banach algebra of real valued continuous functions of $X$ equipped with the supremum norm.
For $n\in \mathbb{N}$, the Bowen metric $d_n$ on $X$  is defined by
\begin{align*}
d_n(x,y):=\max_{0\leq i\leq n-1}d(T^i(x),T^i(y)).
\end{align*}
Let $n\in \mathbb{N}$ and $\epsilon>0$. A subset $E$ of $X$ is said to be
$(n,\epsilon)$ separated with respect to $T$ if
$x,y\in E$, $x\neq y$, implies $d_n(x,y)>\epsilon$.
Let $s(n,\epsilon)$ denote the largest cardinality of $(n,\epsilon)$ separated
subsets of $X$ with respect to $T$.
Define $$P(\varphi, n,\epsilon):= \sup\left\lbrace \sum_{x\in E}e^{S_n\varphi(x)}: E \;\text{is a}\;(n,\epsilon )\; \text{separated set of}\;X \right\rbrace .$$
The classical topological pressure (see \cite{Wal}) of $T$ for $\varphi $ is
\begin{align*}
P_{top}(\varphi ):=\lim_{\epsilon\to0}\liminf_{n\to\infty}\frac{1}{n}\log P(\varphi, n,\epsilon)
=\lim_{\epsilon\to0}\limsup_{n\to\infty}\frac{1}{n}\log P(\varphi, n,\epsilon).
\end{align*}
 Let $M(X, T)$ denote the collection of all $T$-invariant Borel probability measure  and $h_{\mu}(T)$ denote  the measure entropy for the invariant measure $\mu$.
The well known variational principle states that
$$P_{top}(\varphi )=\sup\{h_{\mu}(T)+\int \varphi d\mu : \mu \in M(X, T)\}.$$
When $\varphi =0$, $h_{top}(T):= P_{top}(\varphi )$ denotes the classical topological entropy. Clearly, we have the variational principle as above for the classical topological entropy.

The non-compact topological pressure was introduced by Pesin and and Pitskel \cite{PesPit}. For convenience, we consider an equivalent definition (see \cite{Pes}). Next, we give the definition of non-compact topological pressure.
\begin{defn}{\rm\cite{Pes}}
Suppose $Z\subset X$ is an arbitrary Borel set and $\varphi\in C(X)$. Let $\Gamma_{n}(Z,\epsilon)$ be the collection of all finite or countable covers of $Z$ by sets of the form $B_{m}(x,\epsilon),$ with $m\geq n$. Let $S_{n}\varphi(x):=\sum_{i=0}^{n-1}\varphi(T^{i}x)$. Set
\begin{align*}	M(Z,t,\varphi,n,\epsilon):=\inf_{\mathcal{C}\in\Gamma_{n}(Z,\epsilon)}\left\{\sum_{B_{m}(x,\epsilon)\in\mathcal{C}}\exp (-tm+\sup_{y\in B_{m}(x,\epsilon)}S_{m}\varphi(y))\right\},
\end{align*}
	and
	\begin{align*}
	M(Z,t,\varphi,\epsilon)=\lim_{n\to\infty}M(Z,t,\varphi,n,\epsilon).
	\end{align*}
Then there exists a unique number $P(Z,\varphi,\epsilon)$ such that
$$P(Z,\varphi,\epsilon)=\inf\{t:M(Z,t,\varphi,\epsilon)=0\}=\sup\{t:M(Z,t,\varphi,\epsilon)=\infty\}.$$
	$P(Z,\varphi)=\lim\limits_{\epsilon\to0}P(Z,\varphi,\epsilon)$ is called
	the topological pressure of $Z$ with respect to $\varphi$.
\end{defn}
It is  obvious that the following hold:
\begin{enumerate}
	\item[(1)] $P(Z_1,\varphi)\leq P(Z_2,\varphi)$ for any $Z_1\subset Z_2\subset X$;
	\item[(2)] $P(Z,\varphi)=\sup_i P(Z_i,\varphi)$, where $Z=\bigcup_{i=1}^{\infty} Z_i\subset X$.
\end{enumerate}
For the full space $X$, we have  $P_{top}(\varphi )= P(X, \varphi)$. If $\varphi=0$, then $P(Z, 0)= h^B(T, Z)$ denotes the Bowen topological entropy of the subset  $Z$.
Next, we give the definition of topological dynamical systems with the specification property.

 \begin{defn}
A dynamical system $(X,d,T)$ satisfies the specification property
if for each $\epsilon>0$ there is an integer
$m=m(\epsilon)$ such that for any collection $\{I_{j}=[a_{j},b_{j}]\subset\mathbb N:j=1,\cdots,k\}$ of finite intervals with $a_{j+1}-b_{j}\geq m(\epsilon)$ for $j=1,\cdots,k-1$ and any $x_{1},\cdots,x_{k}$ in $X$, there exists a point $x$ such that
$$d(T^{p+a_{j}}x,T^{p}x_{j})<\epsilon~\text{for all}~p=0,\cdots,b_{j}-a_{j}~\text{and every}~j=1,\cdots,k.$$
\end{defn}

\section{Proof of Theorem \ref{main}}\label{proofm}
In this section, we will prove Theorem \ref{main}. It is obvious that $P(E(z_0  ),\varphi)\leq P_{top}(\varphi)$.
 We only need to prove the lower bound. Inspired by Thompson\cite{Tho1}, we first give the proposition about the Generalised Pressure Distribution Principle.
\begin{prop}
	{\rm \cite[Proposition 3.2]{Tho1}}
	Let $(X,d,T)$ be a
	topological dynamical system. Let $Z\subset X$ be an arbitrary Borel
	set. Suppose there exist $\epsilon>0$ and $s\geq0$ such that one can
	find a sequence of Borel probability measures $\mu_k,$ a constant
	$K>0$ and an integer $N$ satisfying
	\begin{align*}
	\limsup\limits_{k\to\infty}\mu_k(B_n(x,\epsilon))\leq
	K\exp(-ns+\sum\limits_{i=0}^{n-1}\varphi(T^ix))
	\end{align*}
	for every ball $B_n(x,\epsilon)$ such that $B_n(x,\epsilon)\cap
	Z\neq\emptyset$ and $n\geq N.$ Furthermore, assume that at least one
	limit measure $\nu$ of the sequence $\mu_k$ satisfies $\nu(Z)>0.$
	Then $P(Z, \varphi,\epsilon)\geq s.$
\end{prop}
Next, we begin the proof of Theorem \ref{main}.
Let $\epsilon >0$. We set $$\text{Var}(\varphi,\epsilon):= \sup\{|\varphi(x)-\varphi (y)|: d(x, y)\leq \epsilon \}.$$
For any real number $\mathbf C$ satisfy $\mathbf C< P_{top}(\varphi )$, we only need to show that
\begin{align}
\label{zhuyao}
  P(E(z_0),\varphi )\geq \mathbf C.
\end{align}
Firstly, since $z_0$ is non-transitive point,  we can choose choose $y\in X$ and $\epsilon_0 >0$ such that
\begin{align}\label{mao}
d\Big(y,    \overline{O_T(z_0)} \Big)\geq 2\epsilon_0 .
\end{align}
Fix sufficiently small $\eta >0$,
we can choose $\epsilon<\epsilon_0$
\begin{align}\label{suan}
\text{Var}(\varphi,2\epsilon)<\eta.
\end{align}
By the definition of topological pressure
$P_{top}(\varphi)= \lim\limits_{\epsilon\to0}\liminf\limits_{n\to\infty}\dfrac{1}{n}\log P(\varphi, n,\epsilon)$,  we can take a sequence  $\{n_k\}_{k\geq 1} \subset \mathbb N$ such that
 there exists  $(n_k, 9\epsilon)$-separated set $\mathcal{S}_k$ with  maximal  cardinality satisfying
\begin{align}\label{diyi}
\mathscr{M}_k:=\sum\limits_{x\in
\mathcal{S}_k}\exp\left\{\sum\limits_{j=0}^{n_k-1}\varphi(T^jx)\right\}\geq\exp n_k(\mathbf C-\eta),
\end{align}

\subsection{Construction of the Fractal $\mathbf F$}
Choose $M>0$   such that
\begin{align}\label{wucha}
\dfrac{2m(\epsilon )(\mathbf C+\|\varphi \|)}{M+2m(\epsilon )}<\eta.
\end{align}
Without losing generality, we assume that $M<n_1$.
Let $c_k=\lceil \frac{n_k}{M} \rceil $ then we break the $n_k$ orbit of $x\in \mathcal{S}_k$ as follows
$$\{  x, Tx, \cdots, T^{M-1}x\}\cup \{  T^Mx, T^{M+1}x, \cdots, T^{2M-1}x\}\cup \cdots \cup \{  T^{(c_k-1)M}x, T^{(c_k-1)M+1}x, \cdots, T^{n_k-1}x\}. $$
We replace $\{x, Tx, T^2x, \cdots, T^{n-1}\}$ by
$(x, n)$. We insert $(y,1)$ into $(x, n_k)$ as follows:
$$(x, M), (y, 1),  (T^Mx, M),  (y, 1), (T^{2M}x, M), (y, 1)\cdots .$$
 By the specification property, we can set the  following non-empty  set,
\begin{align*}
B(x,  n_k, \epsilon; y)=&
\bigcap_{j=1}^{c_k-1 }
T^{-(j-1)(M+2m(\epsilon )+1)}B_{M}(T^{(j-1)M}x ,\epsilon )\cap T^{ -
	(j-1)(M+2m(\epsilon )+1)-M-m(\epsilon )}B(y, \epsilon )\\
&\cap T^{ -
	(c_k-1)(M+2m(\epsilon )+1)-M-m(\epsilon )}B_{n_k-(c_k-1)M}(T^{(c_k-1)M}x, \epsilon ) \neq \emptyset.
\end{align*}
From the above setting, let $\widehat{n}_k:= n_k+(c_k-1)(2m(\epsilon )+1)$.
Choose a sequence $N_{k}$ and $N_k$ increasing to $\infty$ with $N_{0}=0$. Now we set $t_0=-m(\epsilon), t_1:= \widehat{n}_1N_1+  (N_1-1)m(\epsilon ) $ and  we can inductively define $t_{k+1}:=t_k+N_{k+1}(\widehat{n}_{k+1}+m(\epsilon )).$
We enumerate the points in the sets $\mathcal{S}_k$ provided by (\ref{diyi})  and write them as follows
$$\mathcal{S}_{k}=\{x_i : i=1, 2, \cdots, \# \mathcal{S}_k\}.$$
Let $\underline{x}_k =(x^k_{1} ,\cdots,x^k_{N_k } )\in \mathcal{S}_{k}^{N_k}$ where  $ \mathcal{S}_{k}^{N_k}= \mathcal{S}_{k}\times\cdots\times \mathcal{S}_{k}$. By the specification property, we have
 \begin{align*}
 B( \underline{x}_1,\cdots,\underline{x}_k; y)=&
\bigcap_{i=1}^{k-1} \bigcap_{j=1}^{N_i }
T^{ -t_{i-1}-m(\epsilon )-
 	(j-1)(\widehat{n}_j+m(\epsilon ))}B(x_j,  n_j, \epsilon; y)
 \neq \emptyset.
 \end{align*}
 We define $F_{k}$ by
 \begin{align*}
 F_{k}=\bigcup\{\overline{B( \underline{x}_1,\cdots,\underline{x}_k; y)}:( \underline{x}_1,\cdots,\underline{x}_k)\in \prod_{i=1}^k\mathcal{S}_{i}^{N_i} \}.
 \end{align*}
 Obviously, $F_{k}$ is compact and $F_{k+1}\subset F_{k}$. Define $$F=\bigcap_{k=1}^{\infty}F_{k}.$$
 The above construction implies that for each $p\in F$ shadows the points in $\mathcal{S}_i$ for some $i$ with the bug segments $m(\epsilon )$ by the specification property.
 For any $n>0$ we denotes $n_{rel}$ by the segment of times which shadow the separated points in $\mathcal{S}_i$ for some $i\geq 1$. The following lemma shows that $F\subset E(z_0)$.

 \begin{lem}\label{guidao}
 	For any $x\in F$, $ x\in E(z_0 )$.
 \end{lem}
 \begin{proof}
 	Let $m:=m(\epsilon )$. Since $B_{M+2m}(z_0, \epsilon )$ is an open set,
 	we only need to show that   $O_T(x)\cap B_{M+2m}(z_0, \epsilon )=\emptyset$. In fact, if $O_T(x)\subset X\setminus   B_{M+2m}(z_0, \epsilon )$,
 we have $  \overline{O_T(x)}\subset X\setminus   B_{M+2m}(z_0, \epsilon )$, which implies that
 	$z_0\notin \overline{O_T(x)}.$
 	
Now we assume that  $O_T(x)\cap B_{M+2m}(z_0, \epsilon )\neq  \emptyset$ and without generality, we can choose $T^j x
\in B_{M+2m}(z_0, \epsilon )$. Then $d(T^{j+i}x, T^iz_0)<\epsilon$ for all $i<M+2m$.
By the construction of $F$, for any $k$ with $t_k>>j$,
there exists some $\underline{x}_1,\cdots,\underline{x}_k$ such that $x\in  B( \underline{x}_1,\cdots,\underline{x}_k; y)$.
Hence, we can choose $i< M+2m$ such that $  d(T^{j+q}x, y)<\epsilon $. Then we have
$$d(y, T^iz_0)\leq  d(T^{j+i}x, y)+d(T^{i+q}x, T^iz_0)\leq \epsilon+\epsilon <2\epsilon_0, $$
which contracts with (\ref{mao}).

 \end{proof}
\subsection{Construction of a special sequence of measures $\mu_k$}
 Let $P_k:=\left\lbrace z( \underline{x}_1,\cdots,\underline{x}_k; y)\in B( \underline{x}_1,\cdots,\underline{x}_k; y)| ( \underline{x}_1,\cdots,\underline{x}_k)\in \prod_{i=1}^k\mathcal{S}_{i}^{N_i}\right\rbrace $.
 The following lemma shows that
\begin{align}
\label{llc}
 \#P_k=\prod_{i=1}^k(\#\mathcal{S}_{i})^{N_i}.
\end{align}

\begin{lem}\label{sen}
Let $\underline{x}$ and $\underline{y}$ be distinct elements of $\prod_{i=1}^k\mathcal{S}_{i}^{N_i} $.
Then $z_1=z(\underline{x})$ and $z_2=z(\underline{y})$ are $(t_k,7\epsilon)$ separated points.
 \end{lem}
\begin{proof}
Assume that $ \underline{x}=(\underline{x}_1, \underline{x}_2, \cdots, \underline{x}_k)$ and $\underline{y}=(\underline{y}_1, \underline{y}_2, \cdots, \underline{y}_k)$ 
Then there exists $\underline{x}_i\neq \underline{y}_i$.
Let $\underline{x}_i =(x^i_{1} ,\cdots,x^i_{N_i } )$ and $\underline{y}_i =(y^i_{1} ,\cdots,y^i_{N_i } )$. Without losing generality, we assume
$x^i_q\neq y^i_q$. Then we have for each $0\leq j\leq c_i-1,\; 0\leq s\leq M-1$,
$$d (T^{j(M+2m(\epsilon )+1)+s}  T^{t_{i-1}+(q-1)(m(\epsilon )+\hat{n}_i)+m(\epsilon )}z(\underline{x}), T^{jM+s}x^i_q  )<\epsilon $$
and
$$d (T^{j(M+2m(\epsilon )+1)+s}T^{t_{i-1}+(q-1)(m(\epsilon )+\hat{n}_i)+m(\epsilon )}z(\underline{y}), T^{jM+s}y^i_q  )<\epsilon. $$
Since $x^i_q\neq y^i_q \in \mathcal{S}_i$ are $(n_i, 9\epsilon )$-separated points, there exists some $1\leq \widehat{j}\leq c_i-1, 0\leq \widehat{s}\leq M-1$ such that
$$d(T^{\hat{j}M+\hat{s}}x^i_q, T^{\hat{j}M+\hat{s}}y^i_q)\geq 9\epsilon. $$
Hence,
\begin{align*}
\begin{split}
&d_{t_k}( z_1, z_2)\\
\geq & d_{t_i}(z_1, z_2)\\
\geq& d (T^{\hat{j}(M+2m(\epsilon )+1)+\hat{s}}T^{t_{i-1}+(q-1)(m(\epsilon )+\hat{n}_i)+m(\epsilon )}z(\underline{x}),T^{\hat{j}(M+2m(\epsilon )+1)+\hat{s}}T^{t_{i-1}+(q-1)(m(\epsilon )+\hat{n}_i)+m(\epsilon )}z(\underline{y}))\\
\geq&d(T^{\hat{j}M+\hat{s}}x^i_q, T^{\hat{j}M+\hat{s}}y^i_q)-d (T^{\hat{j}(M+2m(\epsilon )+1)+\hat{s}}T^{t_{i-1}+(q-1)(m(\epsilon )+\hat{n}_i)+m(\epsilon )}z(\underline{x}), T^{\hat{j}M+\hat{s}}x^i_q  )\\
&-d (T^{\hat{j}(M+2m(\epsilon )+1)+\hat{s}}T^{t_{i-1}+(q-1)(m(\epsilon )+\hat{n}_i)+m(\epsilon )}z(\underline{y}), T^{\hat{j}M+\hat{s}}y^i_q  )\\
\geq & 9\epsilon -\epsilon -\epsilon = 7\epsilon.
\end{split}
\end{align*}
So we are done.
\end{proof}
We now define the measures on $F$ which yield the required estimates
for the pressure distribution principle. For each $z\in P_k,$ we
associate a number $\mathcal{L}_k(z)\in (0,\infty).$ Using these
numbers as weights, for each $k,$ we define an atomic measure
centered on $P_k.$ Precisely, if
$z=z(\underline{x}_1,\cdots,\underline{x}_k),$ we define
\begin{align*}
\mathcal{L}_k(z):=\mathcal{L}(\underline{x}_1)\cdots\mathcal{L}(\underline{x}_k),
\end{align*}
where
$\underline{x}_i=(x_1^i,\cdots,x^i_{N_i})
 \in\mathcal{S}_{i}^{N_i}$
and
\begin{align*}
\mathcal{L}(\underline{x}_i):=\prod_{l=1}^{N_i} \exp S_{n_i}\varphi(x^{i}_{l}).
\end{align*}
We define $\nu_k:=\sum\limits_{z\in L_k}\delta_z\mathcal{L}_k(z).$
We normalize $\nu_k$ to obtain a sequence of probability measures
$\mu_k.$ More precisely, we let $$\mu_k:=\frac{1}{\kappa_k}\nu_k,$$
where $\kappa_k$ is the normalizing constant
$$\kappa_k:=\sum\limits_{z\in P_k}\mathcal{L}_k(z)
=\sum\limits_{\underline{x}_1\in\mathcal{S}_1^{N_1}}\cdots
\sum\limits_{\underline{x}_k\in\mathcal{S}_k^{N_k}}\mathcal{L}(\underline{x}_1)\cdots\mathcal{L}(\underline{x}_k)
=\mathscr{M}_1^{N_1}\cdots \mathscr{M}_k^{N_k}.$$
In order to prove the main results of this article, we
present some lemmas.

\begin{lem}
Suppose that $\nu$ is a limit measure of the sequence of probability measure $\mu_k.$ Then $\nu(F)=1.$
\begin{proof}
Suppose $\nu$ is a limit measure of the sequence of probability measures $\mu_k$. Then $\nu=\lim\limits_{k}\mu_{s_k}$ for some $s_k\to\infty$. For some fixed $s$ and all $p\geq 0$, $\mu_{s+p}( F_s)=1$ since $\mu_{s+p}( F_{s+p})=1$ and $  F_{s+p}\subset   F_{s}$. Therefore, $\nu(  F_{s})\geq \limsup\limits_{k\to\infty}\mu_{s_k}(  F_{s})=1.$ It follows that $\nu( F )=\lim\limits_{s\to\infty}\mu(  F_{s})=1.$
\end{proof}
\end{lem}
Next we set
 $b_n$ denote the mistake segment which at most $n$ i.e.,
$$b_n: =n-n_{rel}.$$
Let $\mathcal{B}=B_n(q,\epsilon)$ be an arbitrary ball which
intersects $F.$ Let $k$ be the unique number which satisfies
$t_k\leq n<t_{k+1}.$ Let $\Delta^{k+1}_{j}:=j(\widehat{n}_{k+1} +m(\epsilon )).$ Let $j\in\{0,\cdots,N_{k+1}-1\}$ be the unique number so
\begin{align*}
t_k+\Delta^{k+1}_{j}\leq n<t_k+\Delta^{k+1}_{j+1}.
\end{align*}
We assume that $j\geq1$ and  the simpler case $j=0$  is similar.

\begin{lem}\label{mea}
For any $p\geq1,$ suppose $\mu_{k+p}(\mathcal{B})>0,$ we have
\begin{align*}
\mu_{k+p}(\mathcal{B})\leq\frac{1}{\kappa_k\mathscr{M}_{k+1}^j}\exp\bigg\{&S_n\varphi(q)+2n\text{Var}(\varphi, 2\epsilon)
 +\|\varphi\|b_n \bigg\}.
	\end{align*}
	where $b_n$ denote the length of mistake segment.
\end{lem}

\begin{proof}
Case $p=1.$ Suppose $\mu_{k+1}(\mathcal{B})>0,$ then $P_{k+1}\cap\mathcal{B}\neq\emptyset$.
Let $z=z(\underline{x},\underline{x}_{k+1})\in P_{k+1}\cap\mathcal{B}$, where $\underline{x}=(\underline{x}_1,\cdots,\underline{x}_k)\in\mathcal{S}_{1}^{N_{1}}\times\cdots\times\mathcal{S}_{k}^{N_{k}}$
and $\underline{x}_{k+1}=(x_1^{k+1},\cdots,x_{N_{k+1}}^{k+1})\in\mathcal{S}_{k+1}^{N_{k+1}}$. Let
	\begin{align*}
	\mathcal{C}_{\underline{x};x_1^{k+1},\cdots,x_j^{k+1}}
	=\{z(\underline{x},(y_1^{k+1},\cdots,y_{N_{k+1}}^{k+1}))\in P_{k+1}:y_1^{k+1}=x_1^{k+1},\cdots,y_j^{k+1}=x_j^{k+1}\}.
	\end{align*}
	Suppose that $z^{\prime}=z(\underline{y},\underline{y}_{k+1})\in P_{k+1}\cap\mathcal{B}$. Since $d_{n}(z,z^{\prime})<2\epsilon$,
	we have $\underline{y}=\underline{x}$ and $y_l^{k+1}=x_l^{k+1}$ for $l\in\{1,\cdots,j\}$.
Thus we have
	\begin{align*}
	\nu_{k+1}(\mathcal{B})
	\leq&\sum_{z\in\mathcal{C}_{\underline{x};x_1^{k+1},\cdots,x_j^{k+1}}}\mathcal{L}_{k+1}(z)\\
	=&\mathcal{L}(\underline{x}_1)\cdots\mathcal{L}(\underline{x}_k)
	\prod_{l=1}^{j}\exp S_{n_{k+1}}\varphi(x^{k+1 }_{l})\mathscr{M}_{k+1}^{N_{k+1}-j}.
	\end{align*}
	Case $p>1.$ Similarly, we have
	\begin{align*}
	&\nu_{k+p}(\mathcal{B})\\
	&\leq\mathcal{L}(\underline{x}_1)\cdots\mathcal{L}(\underline{x}_k)
	\prod_{l=1}^{j} \exp S_{ n_{k+1}}\varphi(x^{k+1 }_{l})
	\mathscr{M}_{k+1}^{N_{k+1}-j}\mathscr{M}_{k+2}^{N_{k+2}}\cdots
	\mathscr{M}_{k+p}^{N_{k+p}}.
	\end{align*}
	Combining with  $d_n(z,q)<\epsilon$, one can readily verify that
	\begin{align*}
	&\mathcal{L}(\underline{x}_1)\cdots\mathcal{L}(\underline{x}_k)
	\prod_{l=1}^{j} \exp S_{ n_{k+1}}\varphi(x^{k+1}_{l})
 \leq\exp\bigg\{S_n\varphi(q)+2n\text{Var}(\varphi, 2\epsilon)
	+\|\varphi\|b_n\bigg\},
	\end{align*}
where $b_n$ denote the length of mistake segment.
Since $\mu_{k+p}=\frac{1}{\kappa_{k+p}}\nu_{k+p}$ and $\kappa_{k+p}=\kappa_{k}\mathscr{M}_{k+1}^{N_{k+1}}\cdots \mathscr{M}_{k+p}^{N_{k+p}}$,
	the desired result follows.
\end{proof}

\begin{lem}
	For sufficiently large $n,$
	\begin{align*}
	\limsup\limits_{l\to\infty}\mu_l(B_n(q,\epsilon))
	\leq\exp\bigg\{-n(P_{top}(\varphi)-2\eta)+S_n\varphi(q)+2n\text{Var}(\varphi,2\epsilon)
	\bigg\}.
	\end{align*}
\end{lem}
\begin{proof}
	By (\ref{diyi}), we have
	\begin{align*}
	\kappa_k\mathscr{M}_{k+1}^j&\geq \exp \left((\sum_{l=1}^kN_ln_l)(\mathbf C-\eta)\right)\exp \Big(jn_{k+1}(\mathbf C-\eta)\Big)\\
	&=\exp \left((\sum_{l=1}^kN_ln_l+jn_{k+1})(\mathbf C-\eta)\right)\\
	&= \exp\Big( (n-b_n)(\mathbf C-\eta)\Big).
	\end{align*}
	By Lemma \ref{mea},
	we have
	\begin{align*}
	\begin{split}
	\mu_{k+p}(B_n(q,\epsilon))\leq&\frac{1}{\kappa_k\mathscr{M}_{k+1}^j}\exp\bigg\{S_n\varphi(q)+2n\text{Var}(\varphi,2\epsilon)
	+\|\varphi\|b_n \bigg\}\\
	\leq& \exp\bigg\{-n(\mathbf C-\eta)+S_n\varphi(q)+2n\text{Var}(\varphi,2\epsilon)
	+(\|\varphi\|+\mathbf C-\eta)b_n \bigg\}\\
	\leq & \exp\bigg\{-n(\mathbf C-\eta)+S_n\varphi(q)+2n\text{Var}(\varphi,2\epsilon)
	+(\|\varphi\|+\mathbf C )b_n \bigg\}\\
	\leq & \exp\bigg\{-n(\mathbf C-2\eta)+S_n\varphi(q)+2n\text{Var}(\varphi,2\epsilon)
	  \bigg\}.
	\end{split}
	\end{align*}
The above inequality follows from (\ref{wucha}) that
\begin{align*}
\dfrac{ (\|\varphi\|+\mathbf C ) b_n}{n}\leq  \eta.
\end{align*}
Hence the desired result follows.
\end{proof}
Applying the generalized pressure distribution principle, we have
\begin{align*}
P(E(z_0),\varphi,\epsilon)\geq P(F,\varphi,\epsilon)\geq \mathbf C-2\eta-2\text{Var}(\varphi,2\epsilon).
\end{align*}
Recall that (\ref{suan}) that $\text{Var}(\varphi,2\epsilon)<\eta$, we have
\begin{align*}
P(E(z_0),\varphi,\epsilon)\geq P(F,\varphi,\epsilon)\geq \mathbf C-4\eta.
\end{align*}
Since $\epsilon$ and $\eta$ were arbitrary, we get (\ref{zhuyao}) and finish  the proof of Theorem \ref{main}.
\section{Applications}
Now we present some examples which satisfy the specification property to illustrate Theorem \ref{main}. On the other hands, we use the Bowen equation for pressure formula to compute the Hausdorff dimension of the non-dense subset.
\begin{Example}\label{ex1}
	Given an integer $k>1$. consider the set $\Sigma^{+}_{k}=\{1,\cdots,k\}^{N} $ of sequences
	$$\omega=(i_{1}(\omega)i_{2}(\omega)\cdots),$$
	where $i_{n}(\omega)\in\{1,\cdots,k\}.$  The shift map $\sigma:\Sigma^{+}_{k}\rightarrow\Sigma^{+}_{k}$ is defined by
	$$\sigma(\omega)=(i_{2}(\omega)i_{3}(\omega)\cdots).$$
\end{Example}\label{ex2}
It is well known that any factor of a topological mixing subshift of finite type has the specification property  and thus our main  theorem applies.

\begin{Example}{\rm \cite{Tho1}}
	Fix $I=[0,1]$ and $\alpha\in(0,1)$. The {\bf Manneville-Pomeau }family of maps are given by
	$$f_{\alpha}: I\rightarrow I,  f_{\alpha}(x)=x+x^{1+\alpha}\mod 1.$$
	Considered as a map of $S^{1}$, $f_{\alpha}$ is continuous. Since $f_{\alpha}'(0)=1,$ the system is not uniformly hyperbolic. However, since the Manneville-Pomeau maps are all topologically conjugate to a full shift on two symbols, they satisfy the specification property.
\end{Example}

We should recall that Example \ref{ex1}   have been showed for nondense set in
\cite{Dol} and \cite{Urb} by a symbolic method.  In this paper, the result may be the first time to study this problems on  the pure topological dynamical systems under specification property.

Next, We sates the BS-dimension which is introduced by  Barreira and Schmeling in \cite{BarSch}.
\begin{defn}{\rm\cite{BarSch}}
Let $\varphi : X\to \mathbb R^+$be a strictly positive continuous function. For each $Z\subset X.$
	Let $\Gamma_{n}(Z,\epsilon)$ be the
	collection of all finite or countable covers of $Z$ by sets of the
	form $B_{m}(x,\epsilon),$ with $m\geq n$. Let $S_{n}\varphi(x):=\sum_{i=0}^{n-1}\varphi(T^{i}x)$. Set
	\begin{align*}
	N(Z,t,\varphi,n,\epsilon):=\inf_{\mathcal{C}\in\Gamma_{n}(Z,\epsilon)}\left\{\sum_{B_{m}(x,\epsilon)\in
		\mathcal{C}}\exp (-t \sup_{y\in
		B_{m}(x,\epsilon)}S_{m}\varphi(y))\right\},
	\end{align*}
	and
	\begin{align*}
	N(Z,t,\varphi,\epsilon)=\lim_{n\to\infty}N(Z,t,\varphi,n,\epsilon).
	\end{align*}
	Then there exists a unique number $P(Z,\varphi,\epsilon)$ such that
	$$BS(Z,\varphi,\epsilon)=\inf\{t:N(Z,t,\varphi,\epsilon)=0\}=\sup\{t:N(Z,t,\varphi,\epsilon)=\infty\}.$$
	$BS(Z,\varphi)=\lim\limits_{\epsilon\to0}BS(Z,\varphi,\epsilon)$ is called
	the BS-dimension of $Z$ with respect to $\varphi$.
\end{defn}
By the definitions of topological pressure and BS-dimension, we can get that for any set $Z\subset X,$ the BS-dimension of $Z$ is a unique solution of Bowen's equation $P(Z, -s\varphi )= 0, i.e., s= BS(Z, \varphi ).$
\begin{rem}\label{rem}
	\begin{enumerate}
		\item [(1)] If $\varphi =1$ then $BS(Z, \varphi )= h^B(T, Z)$ for every set $Z\subset X.$
		\item  [(2)] If $\varphi = \log dT$, where $X$ is a $C^r$ Riemann manifold, and $T$ is a $C^{1+\delta}$ conformal expanding map on
		$X$. Then $BS(Z, \varphi )= {\rm dim}_{H}(Z)$,  for every $Z\subset X.$
	\end{enumerate}

\end{rem}

From Theorem \ref{main} and Remark \ref{rem}, we get the following Corollary.

\begin{cro}
 Suppose that $(X, d, T)$ be a dynamical system with specification property, then
for each non-transitive point $z_0\in X$, we have
$$BS(E(z_0), \varphi ) =  \sup\left\{\frac{h_\mu(T)}{\int \varphi d\mu }: \mu \in M(X, T)\right\},$$
  which is the solution of the Bowen's equation  $P_{top}(-s\varphi )=0.$
\end{cro}

By the definition of BS-dimension we have the following results.
\begin{Example}
	Let $f: M\to M$ be a $C^1$ map on a smooth manifold and let $J\subset M$ be a compact $f$-invariant set. We say $f$-expanding on $J$
	and that $J$ is a repeller for $f$ if there exists $c>0$ and $\tau >1$ such that
	$$\|d_x f^n v\|\geq c\tau^n \|v\|$$
	for all  $ v\in T_xM $ and $n\in \mathbb N.$  It is well known that the map $f: J\to J$ is a factor of a topologically mixing one-side subshift of finite type.
\end{Example}
Finally, we give a Bowen formula for Hausdorff dimension of the non-dense set.

\begin{cro}
  Let $X$ be a $C^r$ Riemann manifold and a repeller of an expanding, $C^{1+\delta}$ conformal topological mixing map $T$. Then 	for each non-transitive point $z_0\in X$,
	$${\rm dim}_H(E(z_0))= \sup\left\{\frac{h_\mu(T)}{\int \|d T\|d\mu }: \mu \in M(X, T)\right\},$$
	where ${\rm dim}_H(\cdot)$ denotes the Hausdorff dimension of a set.

\end{cro}

\section*{Acknowledgements}

The work was supported by the
National Natural Science Foundation of China (No. 11971236), China Postdoctoral Science Foundation (No.2016M591873),
and China Postdoctoral Science Special Foundation (No.2017T100384) and the Postgraduate Research $\&$ Practice Innovation Program of Jiangsu Province (No.KYCX24$\_$1790). The work was also funded by the Priority Academic Program Development of Jiangsu Higher Education Institutions.  We would like to express our gratitude to Tianyuan Mathematical Center in Southwest China (No.11826102), Sichuan University and Southwest Jiaotong University for their support and hospitality.

\section*{Availability of Data and Materials}
Not applicable.
\section*{Declarations}
{\bf Conflict of interest} 

There is no conflict of interest.

{\bf Ethical Approval}

Not applicable

\end{document}